\documentclass[11pt]{amsart}
\usepackage{amssymb}
\usepackage{amsmath, amscd}
\usepackage{amsthm}
\usepackage[table,xcdraw]{xcolor}
\usepackage{ulem}
\usepackage{nicematrix}
\usepackage{tikz}
\usepackage{circuitikz}
\usepackage[colorlinks=true, linkcolor=blue,urlcolor=blue]{hyperref}

\newtheorem{theorem}{Theorem}[section]

\newtheorem{proposition}[theorem]{Proposition}
\newtheorem{lemma}[theorem]{Lemma}

\newtheorem{corollary}[theorem]{Corollary}
\theoremstyle{definition}

\newtheorem{example}[theorem]{Example}
\newtheorem{definition}[theorem]{Definition}

\newtheorem{remark}[theorem]{Remark}

\newtheorem{problem}[theorem]{Problem}

\newcommand{\Z}{\mathbb{Z}}
\newcommand{\Q}{\mathbb{Q}}

\newcommand{\Nil}{{\rm Nil}}
\newcommand{\ch}{{\rm char}}

\date\today

\begin{document}

\author[M.H. Bien]{Mai Hoang Bien\textsuperscript{1,2}}
\address{$^1$ Faculty of Mathematics and Computer Science, University of Science, Ho Chi Minh City, Vietnam; $^2$ Vietnam National University,
Ho Chi Minh City, Vietnam}
\email{mhbien@hcmus.edu.vn}
\author[P.V. Danchev]{Peter Vassilev Danchev\textsuperscript{3}}
\address{$^3$ Institute of Mathematics and Informatics, Bulgarian Academy of Sciences, 1113 Sofia, Bulgaria}
\email{danchev@math.bas.bg}
\author[M. Ramezan-Nassab]{Mojtaba Ramezan-Nassab\textsuperscript{4,5}}
\address{$^4$ (Corresponding author) Department of Mathematics, Kharazmi University, South Mofateh St. 15719-14911, Tehran, Iran; $^5$ School of Mathematics, Institute for Research in Fundamental Sciences (IPM), P.O. Box 19395-5746, Tehran, Iran}
\email{ramezann@khu.ac.ir}

\title[Generalized $t$-Fine Rings]{Generalized $t$-Fine and Quasi $t$-Fine Rings}
\keywords{(generalized) fine ring; (generalized) $t$-fine ring; (generalized) quasi $t$-fine ring; matrix ring; group ring}
\subjclass[2020]{16S34, 16S50, 16U60, 16U99}

\begin{abstract}
Fine rings and generalized fine rings have been extensively studied through additive decompositions involving units and nilpotent elements. In this paper, we develop analogous decompositions involving torsion units. We introduce and study {\it generalized $t$-fine} rings, in which every element outside the Jacobson radical is expressed as a sum of a torsion unit and a nilpotent element. We establish several basic properties of these rings and prove, in particular, that when the nilpotent elements form a subring, a ring is generalized $t$-fine if and only if it is local and every unit is torsion. We further investigate the behaviour of this property for matrix rings, endomorphism rings of finite abelian groups, and group rings, obtaining several structural and characterization results. We then introduce the broader class of {\it generalized quasi $t$-fine} rings by replacing nilpotent elements with quasinilpotent elements. We provide examples that illustrate the difficulties in characterizing this class, and investigate its behaviour for matrix rings and group rings.
\end{abstract}

\maketitle
\section{Introduction and Conventions}
In the present paper, all rings under consideration have an identity element, often denoted by $1$. All notation and terminology are standard, and new concepts are defined explicitly below. As usual, for a ring $R$, the Jacobson radical, the set of nilpotent elements, the center, and the group of units are denoted by ${\rm J}(R)$, $\Nil(R)$, $C(R)$, and $\mathcal{U}(R)$, respectively. The subset of $\mathcal{U}(R)$ consisting of torsion units is denoted by $\mathcal{T}(R)$. The ring of all $n\times n$ matrices over $R$ is denoted by $\mathbb{M}_n(R)$.

In \cite{cl}, C\v{a}lug\v{a}reanu and Lam introduced and studied the concept of \textit{fine} rings, defined by the condition $R\setminus \{0\}= \mathcal{U}(R)+\Nil(R)$. This defines a new (proper) class of simple rings that properly contains the class of simple Artinian rings. This notion was further generalized by Zhou in \cite{zhou} to so-called \textit{generalized fine} rings $R$, for which $R\setminus {\rm J}(R) = \mathcal{U}(R)+\Nil(R)$. An example of a generalized fine ring which is \textit{not} fine was also constructed. (Generalized) fine rings and their extensions have been extensively studied by various authors; see, e.g., \cite{czuun, cziun, nzmax, zhou, zhou2}.

In a similar vein, we investigated in \cite{semi-nil} a proper subclass of fine rings, named \textit{$t$-fine rings}, defined as those rings $R$ for which $R\setminus \{0\}= \mathcal{T}(R)+\Nil(R)$. In light of these viewpoints, we now expand $t$-fine rings to two types of their generalized versions, which we briefly review here.

In Section~\ref{sec 2}, we define the notion of a \textit{generalized $t$-fine} ring as a ring $R$ for which $R\setminus {\rm J}(R)=\mathcal{T}(R)+\Nil(R)$. We explore some elementary properties of generalized $t$-fine rings and, as a notable result, show that if $\Nil(R)$ is a subring of $R$, then $R$ is a generalized $t$-fine ring if and only if $R$ is a local ring with $\mathcal{U}(R)$ torsion (Theorem~\ref{NI local}). The matrix ring over a generalized $t$-fine ring $R$ is also shown to be again generalized $t$-fine, provided that either $R$ is a weakly 2-primal ring or an algebra over a field with at least three elements (see Theorems~\ref{matrix1} and~\ref{matrix2}). It is proved that the endomorphism ring of a finite abelian group $G$ is generalized $t$-fine if and only if $G$ is a direct product of $n$ copies of $\mathbb{Z}_{p^k}$ for some positive integers $k$ and $n$ (Theorem~\ref{endo}). Necessary and/or sufficient conditions for a group ring to be generalized $t$-fine are also provided in Theorem~\ref{group ring1} and Corollary~\ref{group ring nilpotent}.

In Section~\ref{sec 3}, we study a wider class than generalized $t$-fine rings, namely \textit{generalized quasi $t$-fine} rings, defined as rings $R$ for which $R\setminus {\rm J}(R)=\mathcal{T}(R)+\mathcal Q(R)$, where $\mathcal Q(R)$ denotes the set of quasinilpotents in $R$. We give a few examples that illustrate the difficulty of obtaining a more satisfactory description of the class of generalized quasi $t$-fine rings. Then we investigate the generalized quasi $t$-fine property in some ring extensions, such as matrix rings (in Proposition~\ref{matrix3}) and group rings (in Proposition~\ref{group ring3}).

\medskip
\section{\texorpdfstring{Generalized $t$-Fine Rings}{Generalized t-Fine Rings}}\label{sec 2}

Recall that for a ring $R$, $\mathcal U(R)+\Nil(R)\subseteq R\setminus{\rm J}(R)$, and if the equality holds, $R$ is called generalized fine (see \cite{zhou}). Motivated by this, we begin with the following definition.

\begin{definition}
A ring $R$ is called {\it generalized $t$-fine} if every element not lying in the Jacobson radical is the sum of a torsion unit and a nilpotent element, equivalently,
$$R\setminus {\rm J}(R)=\mathcal{T}(R)+\Nil(R).$$
\end{definition}

Thus, if ${\rm J}(R)=0$, the generalized $t$-fine property always coincides with $t$-fine property. In particular, any generalized $t$-fine von Neumann regular ring is necessarily $t$-fine.

Observe that \(\mathbb{Q}\) is an example of a (generalized) fine ring which is not (generalized) $t$-fine, while \(\mathbb{Z}_4\) is an example of a generalized $t$-fine ring that is not $t$-fine. Moreover, for a finite field $F$, the power series ring $F[[x]]$ is a generalized fine ring (by \cite[Corollary~2.11]{zhou}) which is neither fine  nor generalized $t$-fine (by Remark~\ref{remark 2.11} below). Thus, the classes of $t$-fine, fine, generalized $t$-fine, and generalized fine rings are related by the following proper inclusions, as illustrated in the diagram below:

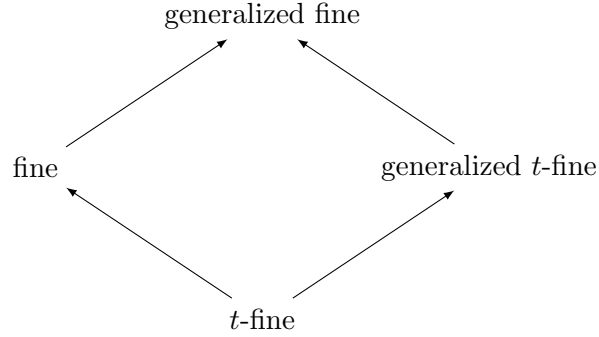
\[
\begin{tikzpicture}[>=latex]
\node (TF) at (0,0) {$t$-fine};
\node (F)  at (-3,2) {fine};
\node (GTF) at (3,2) {generalized $t$-fine};
\node (GF) at (0,4) {generalized fine};

\draw[->] (TF) -- (F);
\draw[->] (TF) -- (GTF);
\draw[->] (F) -- (GF);
\draw[->] (GTF) -- (GF);
\end{tikzpicture}
\]


\medskip

We distribute our establishments into four subsections as presented in the sequel.

\medskip

\subsection{Preliminary Results}\label{sec2}

We start here with some basic results. The set of all unipotent elements in a ring $R$ is hereafter denoted by $\mathfrak{u}(R):=1+\Nil(R)$. So, $\mathfrak{u}(R)\subseteq \mathcal{U}(R)$ is always fulfilled.

\medskip

The following claim is analogous to \cite[Proposition 2.13]{zhou}.

\begin{lemma}
For a ring $R$, the following three statements hold:
\begin{itemize}
\item[(1)] $R$ is $t$-fine if and only if $R$ is generalized $t$-fine and ${\rm J}(R) = 0$.
\item[(2)] $R$ is generalized $t$-fine if and only if $R=(1+{\rm J}(R))\cup(\mathcal{T}(R)+\mathfrak{u}(R))$.
\item[(3)] $R$ is $t$-fine if and only if $R=\{1\}\cup(\mathcal{T}(R)+\mathfrak{u}(R))$.
\end{itemize}
\end{lemma}

\begin{proof}
Point (1) is obvious, so we omit details.

To see (2), choose $r\in R\setminus (1+{\rm J}(R))$. Then, $1-r\not\in {\rm J}(R)$. Since $R$ is generalized $t$-fine, we can write $1-r=u+a$ with $u\in \mathcal{T}(R)$ and $a\in \operatorname{Nil}(R)$. Hence, one inspects that
$$r=(-u)+(1-a)\in \mathcal{T}(R)+\mathfrak{u}(R),$$ as required.

Conversely, choose $x\in R\setminus {\rm J}(R)$. Thus, $1+x\not\in 1+{\rm J}(R)$. Hence, $1+x=v+b$ for some $v\in \mathcal{T}(R)$ and $b\in \mathfrak{u}(R)$. Therefore, $x=v+(b-1)$ and $b-1\in \operatorname{Nil}(R)$. That is why, $R$ is a generalized $t$-fine ring, as needed.

Finally, (3) can be verified similarly to (2).
\end{proof}


\begin{lemma}\label{char}
If $R$ is a generalized $t$-fine ring, then $\operatorname{char}(R)>0$.
\end{lemma}

\begin{proof}
Note that at least one of the elements $3$ and $4$ does not belong to ${\rm J}(R)$. Now, utilizing the proof of \cite[Lemma~1]{semi-nil}, we conclude that $\operatorname{char}(R)>0$, as stated.
\end{proof}


\begin{lemma}\label{unipotent}
If $R$ is a generalized $t$-fine ring, then $\mathfrak{u} (R)\subseteq \mathcal{T}(R)$.
\end{lemma}

\begin{proof}
This follows at once from Lemma~\ref{char} and \cite[Theorem~3.2]{unipotent}.
\end{proof}


The following consequence is somewhat parallel to \cite[Lemma 2.8]{zhou}.

\begin{corollary}\label{center}
If $R$ is a generalized $t$-fine ring, then $C(R)$ is a local ring and $\mathcal{U}(C(R))$ is a torsion group.
\end{corollary}

\begin{proof}
First, note that $\mathcal{U}(C(R))=\mathcal{U}(R)\cap C(R)$. Let $r\in C(R)\setminus {\rm J}(C(R))$. Then, there exists $x\in C(R)$ such that $1-rx$ is not a unit in $C(R)$, and hence it is not a unit in $R$. Therefore, $r\not\in {\rm J}(R)$.

Furthermore, since $R$ is generalized $t$-fine, we can write $r=u+q$, where $u\in \mathcal{T}(R)$ and $q\in \Nil(R)$. Since $r$ is central, $u$ commutes with $q$. Hence, $1+u^{-1}q$ is a unipotent element, and by Lemma~\ref{unipotent} it must be that $1+u^{-1}q\in \mathcal{T}(R)$. Therefore, $r=u(1+u^{-1}q)$ is a product of two commuting torsion elements whence $r\in \mathcal{T}(R)$. Since $r$ is central, it follows that $r\in \mathcal{T}(C(R))$. Consequently, every element of $C(R)\setminus {\rm J}(C(R))$ is a torsion unit, and so  $C(R)$ is local and $\mathcal{U}(C(R))$ is a torsion group, as expected.
\end{proof}


We now explore a more general setting of the above result. First, we record the following auxiliary lemma.

\begin{lemma}\label{local1}
Let $R$ be a local ring of positive characteristic. If $x\in {\rm J}(R)$ and $1+x\in \mathcal{T}(R)$, then $x\in \Nil(R)$.
\end{lemma}

\begin{proof}
Suppose that $\operatorname{char}(R)=p^k$ for some prime $p$ and integer $k\ge1$. Since $1+x\in \mathcal{T}(R)$, there exists a positive integer $m$ such that
\[
(1+x)^m=1.
\]
Write $m=p^rs$, where $r\ge0$, $s\ge1$, and $\gcd(s,p)=1$. Put
\[
y:=(1+x)^s-1.
\]
Then,
\[
(1+y)^{p^r}=\bigl((1+x)^s\bigr)^{p^r}=1.
\]

We first show that $y$ is nilpotent. Notice that
\[
(1+y)^p\equiv 1+y^p \pmod{pR}.
\]
Iterating this congruence yields
\[
1=(1+y)^{p^r}\equiv 1+y^{p^r}\pmod{pR}.
\]
Thus, $y^{p^r}\in pR$.
Hence, there exists $a\in R$ such that $y^{p^r}=pa$. Therefore,
\[
y^{p^rk}=(y^{p^r})^k=(pa)^k=p^ka^k=0,
\]
because $\operatorname{char}(R)=p^k$. So, $y$ is nilpotent.

Next,
\[
y=(1+x)^s-1=xq(x),
\]
where
\[
q(x)=1+(1+x)+(1+x)^2+\cdots +(1+x)^{s-1}=s+j
\]
for some $j\in {\rm J}(R)$. As $p\in {\rm J}(R)$ and $\gcd(s,p)=1$, it must be that $s\not\in {\rm J}(R)$. Hence, $s\in \mathcal{U}(R)$, and so $q(x)\in \mathcal{U}(R)$. It thus follows that $x=yq(x)^{-1}$ is also nilpotent, as asked.
\end{proof}


The following observation is crucial. Recall that a ring $R$ is termed {\it NR}, provided that $\Nil(R)$ forms a subgroup of $R$.

\begin{lemma}\label{NR}
If $R$ is an NR generalized fine ring, then $\Nil(R)\subseteq {\rm J}(R)$.
\end{lemma}

\begin{proof}
Choosing $a\in \Nil(R)\setminus {\rm J}(R)$, we may write $a=u+b$, where $u\in \mathcal{U}(R)$ and $b\in \Nil(R)$, whence $u=a-b\in \Nil(R)$ that is impossible. This substantiates our claim.
\end{proof}


Our first main result is the following.

\begin{theorem}\label{NI local}
Let $R$ be an NR ring. Then, $R$ is a generalized $t$-fine ring if and only if $R$ is a local ring with $\mathcal{U}(R)$ torsion; in this case, ${\rm J}(R)$ is nil.
\end{theorem}

\begin{proof}
Suppose that $R$ is a generalized $t$-fine ring. Let $a\not\in {\rm J}(R)$. Then, $a=u+n$, where $u\in \mathcal T(R)$  and $n\in\Nil(R)$. Since Lemma~\ref{NR} tells us that $\Nil(R)\subseteq {\rm J}(R)$, we have $n\in {\rm J}(R)$. Hence, $a$ is a unit. Thus, $R$ is a local ring.

Now, let $b\in {\rm J}(R)$.  Since $R$ is generalized $t$-fine, we may write $1+b=v+q$, where $v\in \mathcal T(R)$  and $q\in \Nil(R)\subseteq{\rm J}(R)$. 
Set $x:=b-q$. Then, $x\in {\rm J}(R)$ and $1+x=v$ 
is a torsion unit. Owing to Lemmas~\ref{char} and \ref{local1}, $x\in\Nil(R)$. Then, $b=x+q$,
where both $x$ and $q$ are nilpotent elements. But, since $R$ is an NR ring, $b$ is nilpotent. Consequently, ${\rm J}(R)$ is a nil-ideal.

Finally, let $c\in \mathcal{U}(R)$. Write $c=w+s$, where $w\in \mathcal{T}(R)$ and $s\in \Nil(R)$. Suppose now that $w^k=1$ for some positive integer $k$. Then,
\[
c^k=(w+s)^k=1+s'
\]
for some nilpotent element $s'\in {\rm J}(R)$. Thus, $c^k$ is a unipotent unit. Referring to Lemma~\ref{unipotent}, one extracts that $c^k$ is a torsion unit. Hence, $c$ is also a torsion unit, as pursued.  This proves necessity.

The converse is immediate, completing the entire proof.
\end{proof}


We now show the following strengthening of \cite[Lemma~2.10]{zhou}.

\begin{lemma}\label{factor}
Let $R$ be a ring, and let $I$ be an ideal of $R$.
\begin{itemize}
    \item[(i)] If $R$ is a generalized $t$-fine ring, then so is $R/I$.
    \item[(ii)] If $I$ is a nil-ideal and $R/I$ is a generalized $t$-fine ring, then so is $R$.
\end{itemize}
\end{lemma}

\begin{proof}
(i) Let $\bar a\in R/I\setminus {\rm J}(R/I)$. Then, $a\not\in {\rm J}(R)$, and hence there exist $u\in \mathcal{T}(R)$ and $q\in \Nil(R)$ such that $a=u+q$. Therefore, $\bar a=\bar u+\bar q$ is a generalized $t$-fine decomposition in $R/I$, as required.

(ii) Let $a\in R\setminus {\rm J}(R)$. Since $I\subseteq {\rm J}(R)$, we infer $\bar a\not\in {\rm J}(R/I)$. Thus,
\[
\bar a=\bar u+\bar q,
\]
where $\bar u$ is a torsion unit in $R/I$ and $\bar q$ is a nilpotent in $R/I$. Since $I$ is nil, it follows that $q\in \Nil(R)$. Therefore,
$$a=u+(q+b)$$
for some $b\in I$. As $q+b\in\Nil(R)$, it suffices to show that $u\in \mathcal{T}(R)$.

Suppose next that $\bar u^t=\bar1$ for some positive integer $t$. Then, $u^t=1+c$ for some nilpotent element $c\in I$. Employing  Lemma~\ref{unipotent}, we conclude that $u\in \mathcal{T}(R)$, as asked.
\end{proof}


Using the above lemma, we easily derive the following results (compare also with \cite[Corollary~2.11]{zhou}).

\begin{corollary}
The following statements are equivalent:
\begin{enumerate}
    \item \(R\) is a generalized $t$-fine ring;
    \item For every \(R\)-bimodule \(M\), the trivial extension \(T(R,M)\) is a generalized $t$-fine ring;
    \item For every integer \(n \ge 1\), the ring \(R[x;\sigma]/\langle x^n\rangle\) is a generalized $t$-fine ring.
\end{enumerate}
\end{corollary}


\begin{corollary}\label{Jacobson factor}
If $R$ is a generalized $t$-fine ring, then $R/{\rm J}(R)$ is a $t$-fine ring. Conversely, if ${\rm J}(R)$ is nil and $R/{\rm J}(R)$ is a $t$-fine ring, then $R$ is a generalized $t$-fine ring.
\end{corollary}


\begin{remark}\label{remark 2.11}
Let $F$ be a finite field. Consulting with Corollary~\ref{center}, $F[[x]]$ (or $F[x]$) is {\it not} a generalized $t$-fine ring, whereas $F$ is. Now, the isomorphism $F[[x]]/\langle x\rangle\cong F$ unambiguously shows that the assumption that ${\rm J}(R)$ is nil is essential in Corollary~\ref{Jacobson factor}.
\end{remark}
Recall that a {\it PI ring} is a ring that satisfies a polynomial identity. Combining Corollary~\ref{Jacobson factor} with \cite[Proposition~10]{semi-nil}, we obtain the following result.

\begin{corollary}
Let $R$ be a PI ring. If $R$ is generalized $t$-fine, then $R/{\rm J}(R)\cong \mathbb{M}_n(F)$, where $F$ is a locally finite field and $n$ is a positive integer. Conversely, if ${\rm J}(R)$ is nil and $R/{\rm J}(R)\cong \mathbb{M}_n(F)$, where $F$ is a locally finite field, then $R$ is a generalized $t$-fine ring.
\end{corollary}


A ring $R$ is called UU if $\mathcal{U}(R)=1+\Nil(R)$, and WUU if $\mathcal{U}(R)=\pm 1 + \Nil(R)$ (see \cite{unipotent} and \cite{dan1}). Likewise, $R$ is called  JU if $\mathcal{U}(R)=1+{\rm J}(R)$, and WJU if $\mathcal{U}(R)=\pm 1 +{\rm J}(R)$ (see \cite{dan1} and \cite{dan2}).

It is worth mentioning that \cite{cl} proved that a fine UU ring is isomorphic to $\mathbb{F}_2$. Moreover, \cite{zhou} showed that a ring $R$ is generalized fine and UU if and only if $R$ is generalized fine and JU, equivalently, $R/{\rm J}(R)\cong \mathbb{F}_2$.


In this aspect, the following two corollaries have straightforward proofs, which we omit voluntarily.

\begin{corollary}
If $R$ is a ring which is both generalized $t$-fine and UU, then $R/{\rm J}(R)\cong \mathbb{F}_2$. The converse holds whenever ${\rm J}(R)$ is nil.
\end{corollary}


\begin{corollary}
If $R$ is a ring which is both generalized $t$-fine and WJU, then either $R/{\rm J}(R)\cong \mathbb{F}_2$ or $ R/{\rm J}(R)\cong \mathbb{F}_3$. The converse holds whenever ${\rm J}(R)$ is nil.
\end{corollary}

\medskip

\subsection{Matrix Rings}
In this subsection, we explore the generalized $t$-fine property for matrix ring extensions. In \cite[Theorem~2.5]{zhou}, it was shown that if $R$ is a generalized fine ring, then so is $\mathbb M_n(R)$; likewise, in \cite[Theorem~6]{semi-nil}, we proved that $t$-fineness also passes to $\mathbb M_n(R)$. However, we are currently unable to establish the general result that being generalized $t$-fine is inherited by full matrix rings. Nevertheless, in what follows we prove the validity of some interesting partial cases, which are formulated below.

Recall that a ring $R$ is weakly 2-primal if $\Nil(R)$ coincides with the
Levitzki radical of $R$. For example, every commutative ring and every reduced ring are weakly 2-primal.

\begin{theorem}\label{matrix1}
Let $R$ be a weakly $2$-primal generalized $t$-fine ring. Then, for every positive integer $n$, the matrix ring $\mathbb{M}_n(R)$ is also generalized $t$-fine.
\end{theorem}

\begin{proof}
In accordance with Corollary~\ref{Jacobson factor}, the quotient ring $R/{\rm J}(R)$ is a $t$-fine ring. Hence, apply \cite[Theorem~6]{semi-nil} to get that
\[
\mathbb{M}_n(R/{\rm J}(R))
\cong
\mathbb{M}_n(R)/\mathbb{M}_n({\rm J}(R))
\]
is also a $t$-fine ring.

On the other hand, ${\rm J}(R)$ is nil thanks to Theorem~\ref{NI local}. Furthermore, since $R$ is weakly $2$-primal, we have ${\rm J}(R)=\Nil(R)$, which coincides with the locally nilpotent radical (that is, the Levitzki radical) of $R$. Therefore, $\mathbb{M}_n({\rm J}(R))$ is a nil-ideal of $\mathbb{M}_n(R)$. It now follows from Corollary~\ref{Jacobson factor} that $\mathbb{M}_n(R)$ is a generalized $t$-fine ring, as desired.
\end{proof}


Combining Corollary~\ref{center}, Theorem~\ref{NI local} and Theorem~\ref{matrix1}, we immediately obtain the following consequence.

\begin{corollary}
Let $R$ be a commutative ring, and let $n\geq 1$ be an integer. Then, the following statements are equivalent:
\begin{enumerate}
\item $\mathbb{M}_n(R)$ is a generalized $t$-fine ring.
\item $R$ is a generalized $t$-fine ring.
\item $R$ is a local ring and $\mathcal{U}(R)$ is torsion.
\end{enumerate}
\end{corollary}


\begin{example}
For any natural number $n$, the ring $\mathbb{M}_n(\Q)$ is a fine ring (see \cite{cl}) that is {\it not} a generalized $t$-fine ring in conjunction with Lemma~\ref{char}, while the ring $\mathbb{M}_n(\Z_4)$ is a generalized $t$-fine ring that is {\it not} a fine ring (as it is not simple).
\end{example}


For an integer $n>1$ and $1\le i\neq j\le n$, let $E_{ij}$ denote the $n\times n$ matrix whose $(i,j)$-entry is $1$ and whose remaining entries are $0$. For $\alpha\in R$ and $1\le i\neq j\le n$, define
\[
T_{ij}(\alpha):=\mathrm I_n+\alpha E_{ij}.
\]

We use the following facts frequently in the proof of our next result.

\begin{remark}\label{r1}
Let $R$ be a ring.
\begin{enumerate}
\item If $x\in R$ is a generalized $t$-fine element, then $uxu^{-1}$ is also a generalized $t$-fine element for every $u\in \mathcal{U}(R)$.

\item Suppose that $x,y\in R$ satisfy $x-y\in \mathcal{U}(R)$. Then, for any $a\in R$ and $b\not\in {\rm J}(R)$, at least one of the two elements $a+xb$ and $a+yb$ does not belong to ${\rm J}(R)$. Indeed, if both $a+xb$ and $a+yb$ belong to ${\rm J}(R)$, then
\[
(x-y)b
=(a+xb)-(a+yb)\in {\rm J}(R).
\]
Since $x-y$ is invertible, it follows that $ b\in {\rm J}(R)$, a contradiction.
\end{enumerate}
\end{remark}

We are now prepared to prove the following result.

\begin{theorem}\label{matrix2}
Let $F$ be a field with at least three elements, and let $R$ be an $F$-algebra. If $R$ is generalized $t$-fine, then $\mathbb{M}_n(R)$ is also generalized $t$-fine for every positive integer $n$.
\end{theorem}

\begin{proof}
We establish the theorem by induction on $n$.
The case $n=1$ is trivial, so assume that $n>1$, and suppose that every matrix $A\in \mathbb{M}_k(R)$ admits a sum
\[
A=T_A+N_A,
\]
where $T_A\in\mathbb{M}_k(R)$ is torsion and $N_A\in\mathbb{M}_k(R)$ is nilpotent for each $1\le k<n$.

Set
$
A:=(a_{ij})\in\mathbb{M}_n(R)\setminus {\rm J}(\mathbb{M}_n(R)).
$
We shall show that
$
A=T_A+N_A,$
where $T_A$ is torsion and $N_A$ is nilpotent. We proceed by distinguishing several cases as follows.

\medskip

\noindent \textit{Case 1: There exist $1\le i<j\le n$ such that $a_{ii},a_{jj}\not\in {\rm J}(R)$}.

Write
\[
A=
\begin{pmatrix}
B&C\\
D&E
\end{pmatrix},
\]
where $B\in\mathbb{M}_i(R)$, $E\in\mathbb{M}_{n-i}(R)$, $C\in\mathbb{M}_{i\times(n-i)}(R)$ and $D\in\mathbb{M}_{(n-i)\times i}(R)$.
By the induction hypothesis,
\[
B=T_B+N_B,\quad\text{and}\quad
E=T_E+N_E,
\]
where $T_B,T_E$ are torsion matrices and $N_B,N_E$ are nilpotent matrices.
Therefore,
\[
A=
\begin{pmatrix}
T_B&\bf0\\
D&T_E
\end{pmatrix}
+
\begin{pmatrix}
N_B&C\\
\bf0&N_E
\end{pmatrix}.
\]
Set
\[
T_A:=
\begin{pmatrix}
T_B&\bf0\\
D&T_E
\end{pmatrix},
\quad\text{and}\quad
N_A=
\begin{pmatrix}
N_B&C\\
\bf0&N_E
\end{pmatrix}.
\]
Evidently $N_A$ is nilpotent. Since $T_B$ and $T_E$ are torsion, there exists a natural number $m$ such that
\[
T_A^m
=
\begin{pmatrix}
\mathrm I_i&\bf0\\
D'&\mathrm I_{n-i}
\end{pmatrix}.
\]
Hence, $T_A^m$ is unipotent. By virtue of Lemma~\ref{char} and \cite[Theorem~3.2]{unipotent}, it follows that $T_A$ is torsion, as desired.
\medskip

\noindent \textit{Case 2: There exists a unique index $i\in\{1,\ldots,n\}$ such that $a_{ii}\notin {\rm J}(R)$, whereas $a_{jj}\in {\rm J}(R)$ for every $j\neq i$}.

We may conjugate $A$ by an invertible permutation matrix so that $a_{11}$ and $a_{ii}$ are interchanged. Hence, in view of Remark~\ref{r1}(i), we may assume without loss of generality that
$
a_{11}\not\in {\rm J}(R)$ and $a_{ii}\in {\rm J}(R)$ for all $2\le i\le n$.

We further divide this case into three subcases.

\smallskip

\noindent \textit{Subcase 2.1: $a_{21}\not\in {\rm J}(R)$}. In view of Remark~\ref{r1}(ii), and since $F$ has at least three elements, we may choose $0\neq x\in F$ such that
\[
a_{11}+xa_{21}\not\in {\rm J}(R).
\]
Then,
\[
T_{12}(x)AT_{12}(x)^{-1}
=
\begin{pmatrix}
\begin{matrix}
a_{11}+xa_{21}&*\\
a_{21}&a_{22}-a_{21}x
\end{matrix}&\vline &*\\
\hline *&\vline &*
\end{pmatrix}.
\]
But both diagonal entries $a_{11}+xa_{21}$ and $a_{22}-a_{21}x$
lie outside ${\rm J}(R)$. Hence, by Case 1,
$
T_{12}(x)AT_{12}(x)^{-1}
$
is a generalized $t$-fine element. Therefore, so is $A$, in view of Remark~\ref{r1}(i).

\smallskip

\noindent \textit{Subcase 2.2: $a_{12}\not\in {\rm J}(R)$}. The proof is entirely analogous to that of Subcase~2.1 by considering $T_{21}(x)AT_{21}(x)^{-1}$,
where $0\neq x\in F$ satisfies
$
a_{11}-a_{12}x\not\in {\rm J}(R).$

\smallskip

\noindent \textit{Subcase 2.3: $a_{12},a_{21}\in {\rm J}(R)$}. Consider
\[
T_{21}(1)AT_{21}(1)^{-1}
=
\begin{pmatrix}
\begin{matrix}
a_{11}-a_{12}&a_{12}\\
a_{11}-a_{12}+a_{21}-a_{22}&a_{22}+a_{12}
\end{matrix}&\vline &*\\
\hline *&\vline&*
\end{pmatrix}.
\]
Its $(2,1)$-entry is $a_{11}-a_{12}+a_{21}-a_{22}$, which obviously does not belong to ${\rm J}(R)$. Thus, we are reduced to Subcase~2.1. This completes Case 2.

\medskip

\noindent \textit{Case 3: $a_{ii}\in {\rm J}(R)$ for every $1\le i\le n$}.

Since
$
A\not\in {\rm J}(\mathbb{M}_n(R))
=
\mathbb{M}_n({\rm J}(R)),
$
there exist indices $i\neq j$ such that
$
a_{ij}\not\in {\rm J}(R).$
Without loss of generality, assume that $i<j$. Consider
\[
T_{ji}(1)AT_{ji}(1)^{-1}.
\]
Its $(i,i)$- and $(j,j)$-entries are, respectively,
$
a_{ii}-a_{ij}$ and
$a_{jj}+a_{ij}
$. Evidently, both of them lie outside ${\rm J}(R)$. Hence, Remark~\ref{r1}(ii) applies, and we are thereby reduced to Case 1.

Consequently, every possible case leads to the conclusion that $A$ is a generalized $t$-fine element, as required, thus completing the proof.
\end{proof}


Although $\mathbb M_n(\mathbb F_2)$ is $t$-fine by \cite[Theorem~6]{semi-nil}, we are presently unaware of what happens to the generalized $t$-fine structure of $\mathbb M_n(R)$ when $R$ is an algebra over $\mathbb F_2$.

\medskip

\subsection{Endomorphism Ring of Abelian Groups}
It has been proved in \cite[Corollary~12]{semi-nil} that for any abelian group $G$ its endomorphism ring $\operatorname{End}(G)$ is a $t$-fine ring if and only if $G$ is a finite elementary abelian $p$-group for some prime $p$. In the following, we seek a proper generalization.

\begin{remark}\label{End factor}
Let $p$ be a prime number, and let
$$H\cong\prod_{i=1}^k\Z_{p^{r_i}},\quad\text{where}\quad 1\leq r_1\leq\cdots\leq r_k<\infty.$$
Then, bearing in mind \cite[Corollary~20.14]{End} and \cite[Lemma~4.8]{Li}, we deduce that
\[
\operatorname{End}(H)/{\rm J}(\operatorname{End}(H))\cong \prod_{s\geq0}\mathbb{M}_{n_s}(\mathbb F_p),
\]
where $n_s=|\{r_i : r_i > s\}| - |\{r_i : r_i > s + 1\}|$. In particular, if
$$H\cong\prod_{i=1}^t (\Z_{p^{r_i}})^{m_i},\quad\text{where}\quad r_1<\cdots<r_t,$$
then it is easily seen that $n_s=m_i$ if $s=r_i-1$, and $n_s=0$ otherwise. Therefore, in this case, we derive
\[
\operatorname{End}(H)/{\rm J}(\operatorname{End}(H))\cong \prod_{i=1}^t \mathbb{M}_{m_i}(\mathbb F_p).
\]
\end{remark}


We are now able to establish the following.

\begin{theorem}\label{endo}
Let \(G\) be a finite abelian group. Then, $\operatorname{End}(G)$ is a generalized $t$-fine ring if and only if $G\cong (\mathbb Z_{p^k})^n$ for some prime \(p\) and positive integers \(k,n\).
\end{theorem}

\begin{proof}
First, suppose that \(G\cong (\mathbb Z_{p^k})^{n}\). Thus, $\operatorname{End}(G)\cong \mathbb{M}_n(\mathbb Z_{p^k})$,
and hence
\[
\operatorname{End}(G)/{\rm J}(\operatorname{End}(G))
\cong
\mathbb{M}_n(\mathbb Z_{p^k}/p\mathbb Z_{p^k})
\cong \mathbb{M}_n(\mathbb F_p),
\]
which is $t$-fine. Since ${\rm J}(\operatorname{End}(G))$ is nil, Corollary~\ref{Jacobson factor} informs us that $\operatorname{End}(G)$ is a generalized $t$-fine ring.

Conversely, suppose that $\operatorname{End}(G)$ is a generalized $t$-fine ring. Write the primary decomposition of \(G\) as
\[
G=\prod_{p\mid |G|}G_p,
\]
where \(G_p\) denotes the \(p\)-primary component of \(G\). Since $\operatorname{Hom}(G_p,G_q)=0$ for any two distinct primes $p$ and $q$,
we have
\[
\operatorname{End}(G)
\cong
\prod_{p\mid |G|}\operatorname{End}(G_p).
\]
Consequently,
\[
\operatorname{End}(G)/{\rm J}(\operatorname{End}(G))
\cong
\prod_{p\mid |G|}
\operatorname{End}(G_p)/{\rm J}(\operatorname{End}(G_p)).
\]
Since this quotient is $t$-fine, it is simple and, therefore, there can be only one prime divisor of \(|G|\). Hence, \(G\) is a finite abelian $p$-group for some prime number $p$.

Now, write
\[
G\cong
(\mathbb Z_{p^{r_1}})^{m_1}
\times\cdots\times
(\mathbb Z_{p^{r_t}})^{m_t},
\qquad
r_1<\cdots<r_t.
\]
Therefore, Remark~\ref{End factor} allows us to write that
\[
\operatorname{End}(G)/{\rm J}(\operatorname{End}(G))
\cong
\prod_{i=1}^t \mathbb{M}_{m_i}(\mathbb F_p).
\]
So, this quotient is simple exactly when \(t=1\). Finally,
\[
G\cong  (\mathbb Z_{p^k})^{n}
\]
for some \(k,n\geq1\), as formulated.
\end{proof}


\medskip

\subsection{Group Rings}
In what follows, we are concerned with the behaviour of groups rings in the class of generalized $t$-fine (resp., generalized fine) rings. To this purpose, recall that the map \( \varepsilon : RG \rightarrow R \), defined by
$$\varepsilon\left(\sum_{g \in G} a_g g\right) = \sum_{g \in G} a_g,$$
is a ring epimorphism, known as the \textit{augmentation map} of \( RG \). Its kernel, denoted by \( \Delta(RG) \), is called the \textit{augmentation ideal} of \( RG \).

\medskip

We are now ready to prove the following basic assertion.

\begin{theorem}\label{group ring1}
Let $R$ be a ring and let $G$ be a non-trivial group.
\begin{itemize}
\item[(1)] If $RG$ is a generalized $t$-fine (resp., generalized fine) ring, then $R$ is a generalized $t$-fine (resp., generalized fine) ring and $G$ is a $p$-group for some prime $p\in {\rm J}(R)$.
\item[(2)] If $R$ is a generalized $t$-fine (resp., generalized fine) ring and $G$ is a locally finite $p$-group for some prime $p\in \Nil(R)$, then $RG$ is a generalized $t$-fine (resp., generalized fine) ring.
\end{itemize}
\end{theorem}

\begin{proof}
(1) Knowing Lemma~\ref{factor}, the ring $R$ being a homomorphic image of $RG$ is a generalized $t$-fine (resp., generalized fine) ring as well.

We assert that $\Delta(RG)\subseteq {\rm J}(RG)$. Suppose, to the contrary, that there exists $\alpha\in\Delta(RG)\setminus {\rm J}(RG)$.
Then,
\[
\alpha=u+q,
\]
where $u$ is a unit and $q$ is a nilpotent element of $RG$. Applying the augmentation map $\varepsilon$, we obtain
\[
0=\varepsilon(u)+\varepsilon(q).
\]
Hence,
\[
\varepsilon(u)=-\varepsilon(q)\in \mathcal{U}(R)\cap \Nil(R),
\]
which is false. Therefore, $\Delta(RG)\subseteq {\rm J}(RG)$, as asserted.

It now follows from \cite[Proposition~15(i)]{con} that $G$ is a $p$-group for some prime $p\in {\rm J}(R)$.

\medskip

(2) According to \cite[Proposition~16(ii)]{con}, we find that $\Delta(RG)\subseteq \Nil(RG)$.
Since
\[
RG/\Delta(RG)\cong R,
\]
it directly follows from Lemma~\ref{factor} (resp., \cite[Lemma~2.10]{zhou}) that $RG$ is a generalized $t$-fine (resp., generalized fine) ring, as claimed.
\end{proof}
Recall that a group $G$ is {\it locally solvable} if every finitely generated subgroup is solvable. It is known that each locally solvable torsion group is locally finite (see \cite[Theorem~5.4.11]{robinson}). Combining this with Theorems~\ref{NI local} and~\ref{group ring1}, we deduce the following result.

\begin{corollary}\label{group ring nilpotent}
Let $R$ be an NR ring, and let $G$ be a non-trivial locally solvable group. Then, $RG$ is a generalized $t$-fine ring if and only if $R$ is a generalized $t$-fine ring and $G$ is a $p$-group for some prime $p\in \Nil(R)$.
\end{corollary}

\medskip
\section{\texorpdfstring{Generalized Quasi $t$-Fine Rings}{Generalized  Quasi t-Fine Rings}}\label{sec 3}

We now consider a more general version of generalized $t$-finiteness. To this end, for any ring $R$, let the set of all {\it quasinilpotent elements} of $R$ be denoted by $\mathcal Q(R)$, that is,
\[
\mathcal Q(R) = \{ q \in R \mid 1 - xq \in U(R) \text{ for every } x \in R \text{ that commutes with } q \}.
\]
Note that
$\Nil(R)\cup{\rm J}(R)\subseteq \mathcal Q(R)$ and
$\mathcal{U}(R)+\mathcal Q(R)\subseteq R\setminus {\rm J}(R).$

Several interesting concepts have been introduced via quasinilpotents, including generalized Drazin inverses \cite{Kol}, pseudo Drazin inverses \cite{Wong}, quasipolar rings \cite{Yi}, and $UQN$-rings \cite{ramezan2027}.


We thus come to the following notion.

\begin{definition}
A ring $R$ is called {\it generalized quasi $t$-fine} if every element not lying in the Jacobson radical is the sum of a torsion unit and a quasinilpotent, equivalently,
$$R\setminus {\rm J}(R)=\mathcal{T}(R)+\mathcal Q(R).$$
\end{definition}


It is quite clear that each generalized $t$-fine ring is a generalized quasi $t$-fine ring, but we will show below that the converse manifestly fails.

\begin{example}\label{ex 31}
Let $F$ be a finite field, and let $f(x)\not\in J(F[[x]])=\langle x\rangle$. Then, $f(x)=u+xg(x)$ for some $0\neq u\in F$ and $g(x)\in F[[x]]$. Now, $u$ is a torsion unit and $xg(x)$ is a quasinilpotent element in $F[[x]]$. Thus, $F[[x]]$ is a generalized quasi $t$-fine ring that is {\it not} a generalized $t$-fine ring looking at Remark~\ref{remark 2.11}.
\end{example}


We know from Lemma~\ref{char} that each generalized $t$-fine ring has non-zero characteristic. However, the following example shows that a generalized quasi $t$-fine ring may have zero characteristic.

\begin{example}\label{ex 32}
Let $\Z_{(2)}$ be the localization of $\Z$ at the prime number $2$. Thus, $\Z_{(2)}$ is a local ring of characteristic $0$, and it is easy to see that
$$\mathcal{U}(\Z_{(2)})=\pm1 +{\rm J}(\Z_{(2)}).$$
Therefore, $\Z_{(2)}$ is a generalized quasi $t$-fine ring that is {\it not} a generalized $t$-fine ring.
\end{example}


In both Examples~\ref{ex 31} and \ref{ex 32}, the rings are local and commutative but obviously {\it not} fine. Moreover, if $R$ is a commutative ring, then $\mathcal Q(R)={\rm J}(R)$. In the following, we provide a non-commutative non-local fine ring $S$ such that $S$ is generalized quasi $t$-fine and ${\rm J}(S)\subsetneqq\mathcal Q(S)$.

To exhibit such an example, we use the ring presented in \cite[Section~5]{ramezan2027}. Let $F$ be a field, and let $V = F^{\mathbb{N}}$ be a countably infinite-dimensional vector space over $F$. Put
\[
R := \operatorname{End}_F(V) \cong \operatorname{M}_{\mathrm{cf}}(F),
\]
the ring of column-finite matrices over $F$ (each column has only finitely many nonzero entries).

Then, we define the set \( S \) consisting of those matrices \( A \in R \) which satisfy the following condition: there exists a non-negative integer \( n = n_A \) and matrices
\[
A_1, A_2 \in \mathbb{M}_{2^n}(F), \quad \text{and} \quad A_{i+2} \in \mathbb{M}_{2^{n+i}}(F), \quad i = 1, 2, \dots,
\]
such that \( A \) has the form
\[
\begin{pmatrix}
\begin{pmatrix}
\begin{pmatrix}
A_1 & A_2 \\
\mathbf{0} & A_1
\end{pmatrix} & A_3 \\
\mathbf{0} &
\begin{pmatrix}
A_1 & A_2 \\
\mathbf{0} & A_1
\end{pmatrix}
\end{pmatrix} & A_4 & \dots \\
\mathbf{0} &
\begin{pmatrix}
\begin{pmatrix}
A_1 & A_2 \\
\mathbf{0} & A_1
\end{pmatrix} & A_3 \\
\mathbf{0} &
\begin{pmatrix}
A_1 & A_2 \\
\mathbf{0} & A_1
\end{pmatrix}
\end{pmatrix} & \dots \\
\vdots & \vdots
\end{pmatrix}.
\]

\medskip 

To simplify this presentation, we just write \( A=[n_A;A_1,A_2,\ldots] \).

\medskip

Assume now that \( B = [n_B; B_1, B_2, \ldots] \) is another matrix in \( S \). We may, without loss of generality, assume that \( n_A = n_B \). If, for instance, \( n_A < n_B \), we may reinterpret \( A = [n_B; A'_1, A'_2, \ldots] \) as having the same recursive structure as \( B \) by defining a new matrix \( A'_1 \) as the top-left \( 2^{n_B} \times 2^{n_B} \) block of \( A \).

Thus, it is clear that \( S \) is closed under addition and multiplication, and hence \( S \) is a subring of \( R \).

\medskip

We now proceed by proving the following assertion.


\begin{example}\label{ex 33}
The ring \( S \) satisfies the following conditions:
\begin{itemize}
\item[(1)] ${\rm J}(S)=0$ and $S$ is not local;
\item[(2)] for each field $F$, $S$ is a fine ring; and
\item[(3)] if $F$ is a locally finite field, then $S$ is a (generalized) quasi $t$-fine ring; so each non-zero element of $S$ is a sum of a torsion unit and a quasinilpotent.
\end{itemize}
\end{example}

\begin{proof}
It has been shown in \cite[Example~5.1]{ramezan2027} that $J(S)=0$, and so obviously $S$ is not local.

Suppose \( A \in S \) is nonzero. Then, there exist two indices \( i, j \geq 1 \) such that \( a_{ij} \neq 0 \). We may choose \(n= n_A \) large enough such that \( a_{ij} \) lies in the block \( A_1 \in \mathbb{M}_{2^{n}}(F) \), where \( A=[n;A_1,A_2,\ldots] \).

As \(\mathbb{M}_{2^{n}}(F) \) is a fine ring (see \cite[Theorem]{cl}), we can write \( A_1 = U_1+N_1 \) for some invertible matrix \( U_1 \in \mathbb{M}_{2^{n}}(F) \) and some nilpotent matrix \( N_1 \in \mathbb{M}_{2^{n}}(F) \). Set
$$U:=[n;U_1,A_2,A_3,\ldots],\quad\text{and}\quad N:=A-U=[n;N_1,0,0,\dots].$$
Now, it is readily seen that $U\in \mathcal U(S)$ and $N\in\Nil(S)$. Therefore, $S$ is a fine ring, as asserted.

Now, suppose that $F$ is a locally finite field. Then, by \cite[Corollary~11]{ramezan2027}, \(\mathbb{M}_{2^{n}}(F) \) is a $t$-fine ring. Therefore, in the above decomposition \( A_1 = U_1+N_1 \), we may assume that the matrix $U_1$ is torsion. This time, let
\[
V := [n;U_1,0,0,\dots],\quad\text{and}\quad Q=A-V=[n;N_1,A_2,A_3,\ldots].
\]
It is clear that \( V\in\mathcal T(S)\). We will show now that $Q\in\mathcal Q(S)$.

In fact, suppose \( X=[n_X;X_1,X_2,\ldots] \in S \) satisfies \( QX = XQ \). Since \( N_1^k=0 \) for some positive integer $k$, \( Q^k \) is strictly upper triangular. Therefore, \( Q^{k \cdot n_X} \) is also strictly upper triangular and satisfies
\[
\left[Q^{k \cdot n_X}\right]_{ij} = 0 \quad \text{for all } j - i < n_X.
\]
This gives that \( (XQ)^{k \cdot n_X} = X^{k \cdot n_X}  Q^{k \cdot n_X} \) is strictly upper triangular. It thus follows that \( I - (XQ)^{k \cdot n_X} \) is invertible. Hence, \( I - XQ \) is invertible too. Consequently, \( Q \) is quasinilpotent, as required. Hence, $S$ is a (generalized) quasi $t$-fine ring, as claimed.
\end{proof}


It is known that if $F$ is a field, then $\mathcal Q(\mathbb M_n(F))=\Nil(\mathbb M_n(F))$ (see, e.g., \cite[Example~2.2]{Cui}). In the following result, we seek a generalization of this replacing $F$ with a commutative ring $R$.

\begin{lemma}\label{lem:max}
Let $R$ be a commutative ring. For $A\in \mathbb M_n(R)$, the following statements are equivalent:
\begin{enumerate}
\item[(i)] $A$ is quasinilpotent;
\item[(ii)] For every maximal ideal $\mathfrak m$ of $R$, the matrix $\overline{A}\in \mathbb  M_n(R/\mathfrak m)$ is nilpotent.
\end{enumerate}
\end{lemma}

\begin{proof}
(i) $\Rightarrow$ (ii). Assume that $A$ is quasinilpotent, and let $\mathfrak m$ be a maximal ideal of $R$. Put $F:=R/\mathfrak m$. Suppose, to the contrary, that $\overline{A}$ is not nilpotent in $\mathbb M_n(F)$, and let $p(x)\in F[x]$ be the minimal polynomial of $\overline{A}$. Then, $p(x)\neq x^k$ for every positive integer $k$. Hence, $p(x)$ has an irreducible factor $f(x)$ such that $f(0)\neq0$. Since $\gcd(f(x),x)=1$, there exist $g(x),h(x)\in F[x]$ such that
\[
f(x)g(x)+xh(x)=1.
\]

Choose a polynomial $h_1(x)\in R[x]$ whose reduction modulo $\mathfrak m$ is $h(x)$, and put
\[
X:=h_1(A).
\]
Then, $X$ commutes with $A$. Since $A$ is quasinilpotent, ${\rm I}_n-AX$ is invertible in $\mathbb M_n(R)$. Therefore,
\[
\overline{{\rm I}_n-AX}
=\overline{{\rm I}_n}-\overline{A}h(\overline{A})
=f(\overline{A})g(\overline{A})
\]
is invertible in $\mathbb M_n(F)$. Since the two factors commute each other, both $f(\overline{A})$ and $g(\overline{A})$ are invertible. However, $f$ is a proper factor of the minimal polynomial $p$ of $\overline{A}$, so $f(\overline{A})$ cannot be invertible. This contradiction shows that $\overline{A}$ is nilpotent, as asked.

\medskip

(ii) $\Rightarrow$ (i). Assume that $\overline{A}$ is nilpotent for every maximal ideal $\mathfrak m$ of $R$. Let $X\in \mathbb M_n(R)$ commute with $A$. For an arbitrary maximal ideal $\mathfrak m$, we have
$\overline{X}\overline{A}
=
\overline{A}\overline{X}$.
Since $\overline{A}$ is nilpotent, the matrix $\overline{A}\overline{X}$ is also nilpotent. Hence, $\overline{{\rm I}_n}-\overline{A}\overline{X}$
is invertible in $\mathbb M_n(R/\mathfrak m)$. Consequently,
\[
\det(\overline{{\rm I}_n}-\overline{A}\overline{X})\neq\overline0,
\]
and, therefore,
\[
\det({\rm I}_n-AX)\notin\mathfrak m.
\]
Since $\mathfrak m$ was arbitrary, $\det({\rm I}_n-AX)$ belongs to no maximal ideal of $R$. Thus, it is a unit in $R$ and, consequently, ${\rm I}_n-AX$ is invertible in $\mathbb M_n(R)$. Hence, $A$ is quasinilpotent, as wanted.
\end{proof}


The following claim may be known, but since we are unaware of it appearing in the literature, we include a proof for completeness and the reader's convenience.

\begin{lemma}\label{quasi matrix}
Let $R$ be a commutative ring, and let
\[
A:=
\begin{pmatrix}
B&C\\
\bf 0&D
\end{pmatrix}
\in \mathbb M_n(R)
\]
be a block upper-triangular matrix, where $B\in \mathbb M_r(R)$ and $D\in \mathbb M_s(R)$ with $r+s=n$. If both $B$ and $D$ are quasinilpotent, then $A$ too is quasinilpotent.
\end{lemma}

\begin{proof}
Let $\mathfrak m$ be a maximal ideal of $R$. So, Lemma~\ref{lem:max} works to get that the matrices $\overline{B}$ and $\overline{D}$ are nilpotent in $\mathbb M_r(R/\mathfrak m)$ and $\mathbb M_s(R/\mathfrak m)$, respectively. Therefore,
\[
\overline{A}
=\begin{pmatrix}
\overline{B}&\overline{C}\\
\bf0&\overline{D}
\end{pmatrix}
\in \mathbb M_n(R/\mathfrak m)
\]
is too nilpotent. Consequently, by another application of Lemma~\ref{lem:max}, $A$ is quasinilpotent in $\mathbb M_n(R)$, as desired.
\end{proof}


The following lemma can be proved similarly to Corollary~\ref{center}, so we skip its verification.

\begin{lemma}\label{center2}
If $R$ is a generalized quasi $t$-fine ring, then $C(R)$ is a local ring.
\end{lemma}


Let $R$ be a ring. If $a\in \mathcal Q(R)$ and $u\in\mathcal U(R)$, then $uau^{-1}\in \mathcal Q(R)$ (see, e.g., \cite[Lemma~2.3]{Cui}). Therefore, if $x\in R$ is a generalized quasi $t$-fine element, then $uxu^{-1}$ is also a generalized quasi $t$-fine element for every $u\in \mathcal{U}(R)$. Combining this observation with Lemma~\ref{quasi matrix}, Lemma~\ref{center2} and an argument similar to that of the proof of Theorem~\ref{matrix2}, we deduce the following necessary and sufficient condition.

\begin{proposition}\label{matrix3}
Let $F$ be a field with at least three elements of positive characteristic, and let $R$ be a commutative $F$-algebra. Then, $\mathbb{M}_n(R)$ is generalized quasi $t$-fine if and only if $R$ is generalized quasi $t$-fine for every positive integer $n$.
\end{proposition}


Before presenting our final results, we need the following expected technicality.

\begin{lemma}\label{factor2}
Let $R$ be a ring of non-zero characteristic, and let $I$ be a nil-ideal of $R$. If $R/I$ is a generalized quasi $t$-fine ring, then so is $R$.
\end{lemma}

\begin{proof}
Choose $a\in R\setminus {\rm J}(R)$. Since $I\subseteq {\rm J}(R)$, we can write $\bar a\not\in {\rm J}(R/I)$. Thus,
\[
\bar a=\bar u+\bar q,
\]
where $\bar u\in \mathcal T(R/I)$  and $\bar q\in\mathcal Q(R/I)$.  Therefore,
$$a=u+(q+b)$$
for some $b\in I$. An appeal to \cite[Lemma~4.1(2)]{djhm} leads to $q+b\in\mathcal Q(R)$. Suppose next that $\bar u^t=\bar1$ for some positive integer $t$. Then, $u^t=1+c$ for some nilpotent element $c\in I$. Employing  \cite[Theorem~3.2]{unipotent}, we conclude that $u\in \mathcal{T}(R)$, as required.
\end{proof}


Our final result and its consequence sound thus.

\begin{proposition}\label{group ring3}
Let $R$ be a ring, and let $G$ be a non-trivial group.
\begin{itemize}
\item[(1)] If $R$ is a generalized quasi $t$-fine ring and $G$ is a locally finite $p$-group for some prime $p\in \Nil(R)$, then $RG$ is a generalized quasi $t$-fine ring.
\item[(2)] If $R$ is commutative and $RG$ is a generalized quasi $t$-fine ring, then $R$ is a generalized quasi $t$-fine ring and $G$ is a $p$-group for some prime $p\in {\rm J}(R)$.
\end{itemize}
\end{proposition}

\begin{proof}
(1) This follows from an argument analogous to that in the proof of Theorem~\ref{group ring1}(2) accomplished with Lemma~\ref{factor2}.

(2) Since \( R \) is commutative, it is easy to see that
$$\varepsilon(\mathcal Q(RG)) \subseteq \mathcal Q(R).$$

Suppose \( a \in R\setminus{\rm J}(R) \). Then, there exist \( u \in \mathcal T(RG) \) and \( q \in\mathcal Q(RG) \) such that \( a = u+q \). Hence,
$$ a = \varepsilon(u)+\varepsilon(q) \in \mathcal T(R)+\mathcal Q(R).$$
Therefore, \( R \) is a generalized quasi $t$-fine ring.

Now, as in the proof of Theorem~\ref{group ring1}(1) we can show that $\Delta(RG)\subseteq {\rm J}(RG)$, and so $G$ is a $p$-group for some prime $p\in {\rm J}(R)$, as promised.
\end{proof}


With the help of Theorem~\ref{group ring3} and \cite{local}, we obtain the following criterion.

\begin{corollary}
Let $F$ be a field, and let $G$ be a non-trivial locally finite group. Then, $FG$ is a generalized quasi $t$-fine ring if and only if $F$ is a locally finite field and $G$ is a $p$-group for some prime $p=\ch(F)$. In this case, $FG$ is a local ring.
\end{corollary}

One could define a \textit{quasi $t$-fine ring} to be a ring $R$ such that each non-zero element can be written as a sum of a torsion unit and a quasinilpotent. In this case, ${\rm J}(R)=0$, since if $x \in {\rm J}(R)$ is non-zero, we may write $x = u + q$, where $u \in \mathcal{T}(R)$ and $q \in \mathcal{Q}(R)$. Then $q = x - u$ is invertible, a contradiction. For example, the ring $S$ presented in Example~\ref{ex 33}(3) is a quasi $t$-fine ring.

Obviously, each $t$-fine ring is a quasi $t$-fine ring, but we do not know whether there exists a quasi $t$-fine ring that is not $t$-fine. Thus, we conclude this paper with the following problem, whose answer is currently unknown to us.

\begin{problem}
Does there exist a quasi $t$-fine ring that is not fine, and, if not, does there exist a generalized quasi $t$-fine ring that is not generalized fine?
\end{problem}

\bigskip

\noindent{\bf Acknowledgments.} The first author (Mai Hoang Bien) is funded by the Vietnam National Foundation for Science and Technology Development (NAFOSTED) under grant number 101.04-2025.41. The research work of M. Ramezan-Nassab is supported in part by a grant from IPM (Grant No. 1405160117).

\bigskip

\noindent{\bf Declarations.} Our statements here are the following ones:

\begin{itemize}
\item {\bf Ethical Declarations and Approval:} The authors have no competing interests to declare that are relevant to the content of this article.

\item {\bf Competing Interests:} The authors declare no any conflict of interest.

\item {\bf Availability of Data and Materials:} Data sharing is not applicable to this article as no data-sets or any other materials were generated or analyzed during the current study.
\end{itemize}


\end{document}